\documentclass[11pt]{article}%
\usepackage{amsmath}
\usepackage{amsfonts}
\usepackage{amssymb}
\usepackage{graphicx}%
\providecommand{\U}[1]{\protect\rule{.1in}{.1in}}
\newtheorem{theorem}{Theorem}

\newtheorem{lemma}[theorem]{Lemma}

\newtheorem{proposition}[theorem]{Proposition}
\newtheorem{remark}[theorem]{Remark}

\newenvironment{proof}[1][Proof]{\noindent\textbf{#1.} }{\ \rule{0.5em}{0.5em}}
\begin{document}

\title{A bound for the order of the fundamental group of a complete noncompact Ricci
shrinker$\vspace{-0.06in}$}
\author{Bennett Chow\thanks{Dept. of Math., UC San Diego, La Jolla, CA\ 92093}
\and Peng Lu\thanks{Dept. of Math., Univ. of Oregon, Eugene, OR 97403. P.\thinspace
L. is partially supported by a grant from Simons Foundation.}}
\date{}
\maketitle

$\vspace{-0.4in}$

For shrinking (gradient) Ricci solitons, we note a quantification of Wylie's
\cite{Wylie} result that $\pi_{1}(\mathcal{M})\!<\!\infty$.$\vspace{-0.06in}$

\begin{lemma}
[$f$-uniqueness]Let $(\mathcal{M}^{n},g,f_{i},\varepsilon)$, $i=1,2$, be a
complete gradient Ricci soliton \emph{(}GRS\emph{),} i.e., $\operatorname{Rc}%
+\nabla^{2}f_{i}+\frac{\varepsilon}{2}g=0$. Then $f_{1}-f_{2}$ is constant or
$(\mathcal{M},g)$ is isometric to $(\mathbb{R},ds^{2})\times(\mathcal{N}%
^{n-1},h)$, where $(\mathcal{N},h)$ is isometric to each level set
$\{f_{1}-f_{2}=c\}$, $c\in\mathbb{R}$. Moreover, $(f_{1}-f_{2})(s,y)=as+b$,
$a,b\in\mathbb{R}$.$\vspace{-0.06in}$
\end{lemma}

\begin{proof}
Let $F=f_{1}-f_{2}$. Then $\nabla^{2}F=0$. Assume that $F$ is not a constant.
Then $\left\vert \nabla F\right\vert =a$, where $a\in\mathbb{R}^{+}$. Let
$\Sigma_{c}=\{F=c\}$. Then $\Sigma_{c}$ is $C^{\infty}$ and its second
fundamental form is $\operatorname{II}(X,Y)=\frac{\nabla^{2}F(X,Y)}{|\nabla
F|}=0$. Let $\{\varphi_{t}\}_{t\in\mathbb{R}}$ be the $1$-parameter group of
isometries generated by $\nabla F$. We have $F\circ\varphi_{t}=F+a^{2}t$.
Moreover, $\varphi_{t}$ maps $\Sigma_{c}$ isometrically onto $\Sigma
_{c+a^{2}t}$. \thinspace This produces the product structure on $(\mathcal{M}%
,g)$.\thinspace\thinspace\thinspace\thinspace\thinspace\thinspace
\end{proof}

Let $(\mathcal{M}^{n},g,f)$ be a complete noncompact shrinking GRS, where
$\operatorname{Rc}+\nabla^{2}f=\frac{1}{2}g$ and $f$ is normalized by
$R+|\nabla f|^{2}=f$. Let $(\mathcal{\tilde{M}}^{n},\tilde{g},\tilde{f})$ be
the universal covering shrinking GRS, that is, $\pi:\mathcal{\tilde{M}%
}\rightarrow\mathcal{M}$ is the universal covering map, $\tilde{g}=\pi^{\ast
}g$, and $\tilde{f}=f\circ\pi$. Then we have that $\operatorname{Rc}%
_{\tilde{g}}+\nabla_{\tilde{g}}^{2}\tilde{f}=\frac{1}{2}\tilde{g}$ and
$R_{\tilde{g}}+|\nabla\tilde{f}|_{\tilde{g}}^{2}=\tilde{f}$.

Let $\gamma\in\pi_{1}(\mathcal{M})-\{e\}$. Then $\gamma$ corresponds to a deck
transformation $\gamma:\mathcal{\tilde{M}}\rightarrow\mathcal{\tilde{M}}$ that
is a fixed-point-free isometry of the metric $\tilde{g}$. Since both
$\tilde{f}$ and $\tilde{f}\circ\gamma$ are normalized potential functions for
$\tilde{g}=\gamma^{\ast}\tilde{g}$, by the lemma we have $\tilde{f}\circ
\gamma=\tilde{f}$. In particular, $\gamma$ is a fixed-point-free isometry of
any level or sublevel set of $\tilde{f}$. Hence $\operatorname{Vol}_{\tilde
{g}}(\{\tilde{f}\leq s\})=|\pi_{1}(\mathcal{M})|\operatorname{Vol}_{g}(\{f\leq
s\})$ for any $s>0$. Let $O$ be a minimum point of $f$. By Cao and Zhou
\cite{CaoZhou} and its improvement of constants by Haslhofer and M\"{u}ller
\cite{HaslhoferMuller}, we have$\vspace{-0.06in}$%
\begin{equation}
\frac{1}{4}\left(  (d\left(  x,O\right)  -5n)_{+}\right)  ^{2}\leq f\left(
x\right)  \leq\frac{1}{4}(d(x,O)+\sqrt{2n})^{2}\vspace{-0.06in}%
\label{f Lower Bound with 35 Over 8}%
\end{equation}
for $g$ and the same inequalities for $\tilde{g}$ (using $\tilde{f}$ and using
a lift $\tilde{O}$ of $O$), and $\operatorname{Vol}_{\tilde{g}}(\{\tilde
{f}\leq s\})\leq C(n)s^{n/2}$. Similarly, Munteanu and Wang
\cite{MunteanuWangDensities} proved that $\operatorname{Vol}B_{r}(O)\leq
\omega_{n}e^{n/2}r^{n}$ for $r>0$, where $\omega_{n}=\operatorname{Vol}%
_{\mathbb{R}^{n}}B_{1}$. By Munteanu and Wang \cite{MunteanuWangWeighted},
there is a constant $c(n,\int_{\mathcal{M}}e^{-f}d\mu)>0$ such that
$\operatorname{Vol}B_{r}(O)\geq cr$ for $r\geq1$.$\vspace{-0.06in}$

\begin{proposition}
If $(\mathcal{M}^{n},g,f)$ is a complete noncompact shrinking GRS, then
$|\pi_{1}(\mathcal{M})|\leq C(n,\int_{\mathcal{M}}e^{-f}d\mu)$.$\vspace
{-0.06in}$
\end{proposition}

\begin{proof}
By the above, we have that $C(n)(2n)^{\frac{n}{2}}\!\geq\!\operatorname{Vol}%
_{\tilde{g}}(\{\tilde{f}\!\leq\!2n\})\!=\!|\pi_{1}(\mathcal{M}%
)|\operatorname{Vol}_{g}(\{f\!\leq\!2n\})$. The proposition follows from this
and the inequality $\operatorname{Vol}_{g}(\{f\!\leq\!2n\})\geq
\operatorname{Vol}_{g}(B_{\sqrt{2n}}(O))\geq c(n,\int_{\mathcal{M}}e^{-f}%
d\mu)$.$\vspace{-0.06in}$
\end{proof}

\begin{remark}
Assume for $x\in\mathcal{M}$ and $r>0$ such that $R\leq r^{-2}$ in
$B_{r}\left(  x\right)  $, we have $\operatorname{Vol}B_{s}\left(  x\right)
\geq\kappa s^{n}$ for $0<s\leq r$. \emph{(}For shrinking GRS that are
singularity models, there is such a $\kappa>0$ by Perelman
\emph{\cite{Perelman1}.)} Since $R\leq2n$ on $B_{\frac{1}{\sqrt{2n}}}(O)$, we
have $\operatorname{Vol}B_{\frac{1}{\sqrt{2n}}}(O)\!\geq\!(2n)^{-\frac{n}{2}%
}\kappa$ and we conclude that $|\pi_{1}(\mathcal{M})|\leq C(n)\kappa^{-1}%
$.$\vspace{-0.07in}$
\end{remark}


\begin{thebibliography}{9}                                                                                                %
\bibitem {CaoZhou}Cao, H.-D.; Zhou, D.-T. \emph{On complete gradient shrinking
Ricci solitons.} J. Differential Geom. \textbf{85} (2010), 175--186.\vspace
{-0.06in}

\bibitem {HaslhoferMuller}Haslhofer, R.; M\"{u}ller, R. \emph{A compactness
theorem for complete Ricci shrinkers.} Geom. Funct. Anal. \textbf{21} (2011),
1091--1116.\vspace{-0.06in}

\bibitem {MunteanuWangWeighted}Munteanu, O.; Wang, J. \emph{Analysis of
weighted Laplacian and applications to Ricci solitons.} Comm. Anal. Geom.
\textbf{20} (2012), 55--94.\vspace{-0.06in}

\bibitem {MunteanuWangDensities}Munteanu, O.; Wang, J. \emph{Geometry of
manifolds with densities.} Adv. Math. \textbf{259} (2014), 269--305.\vspace
{-0.06in}

\bibitem {Perelman1}Perelman, G. \emph{The entropy formula for the Ricci flow
and its geometric applications.} arXiv:0211159.\vspace{-0.06in}

\bibitem {Wylie}Wylie, W. \emph{Complete shrinking Ricci solitons have finite
fundamental group.} Proc. Amer. Math. Soc. \textbf{136} (2008), 1803--1806.
\end{thebibliography}
\end{document}